\documentclass[12pt]{article}
\usepackage{amssymb}
\usepackage{amsthm}
\usepackage{amsmath}
\usepackage{graphicx}
\usepackage{enumerate}
\usepackage{pict2e}
\usepackage{framed}

\DeclareMathOperator{\Cl}{Cl}

\DeclareMathOperator{\Con}{Con}

\DeclareMathOperator{\BDed}{\mathbf{Ded}}

\newtheorem{theorem}{Theorem}[section]
\newtheorem{definition}[theorem]{Definition}
\newtheorem{lemma}[theorem]{Lemma}
\newtheorem{proposition}[theorem]{Proposition}

\newtheorem{example}[theorem]{Example}
\newtheorem{corollary}[theorem]{Corollary}
\title{Operators on complemented lattices}
\author{Ivan~Chajda and Helmut~L\"anger$^1$}
\date{}

\begin{document}
	
\footnotetext[1]{Corresponding author}

\footnotetext[2]{Support of the research of the first author by the Czech Science Foundation (GA\v CR), project 24-14386L, entitled ``Representation of algebraic semantics for substructural logics'', and by IGA, project P\v rF~2024~011, is gratefully acknowledged.}
	
\maketitle
	
\begin{abstract}
The present paper deals with complemented lattices where, however, a unary operation of complementation is not explicitly assumed. This means that an element can have several complements. The mapping $^+$ assigning to each element $a$ the set $a^+$ of all its complements is investigated as an operator on the given lattice. We can extend the definition of $a^+$ in a natural way from elements to arbitrary subsets. In particular we study the set $a^+$ for complemented modular lattices, and we characterize when the set $a^{++}$ is a singleton. By means of the operator $^+$ we introduce two other operators $\to$ and $\odot$ which can be considered as implication and conjunction in a certain propositional calculus, respectively. These two logical connectives are ``unsharp'' which means that they assign to each pair of elements a non-empty subset. However, also these two derived operators share a lot of properties with the corresponding logical connectives in intuitionistic logic or in the logic of quantum mechanics. In particular, they form an adjoint pair. Finally, we define so-called deductive systems and we show their relationship to the mentioned operators as well as to lattice filters.
\end{abstract}
	
	

\section{Introduction}

Let $\mathbf L=(L,\vee,\wedge,0,1)$ be a bounded lattice and $a\in L$. An element $b$ of $L$ is called a {\em complement} of $a$ if $a\vee b=1$ and $a\wedge b=0$. The lattice $\mathbf L$ is called {\em complemented} if any of its elements has a complement.

Often lattices with an additional unary operation, usually denoted by $'$, are studied where for each $a\in L$ the element $a'$ denotes its complement. In such a case this unary operation is called a {\em complementation}. However, in complemented lattices we do not assume the complement being unique. This is the case with our present paper.

It is worth noticing that in a distributive complemented lattice the complement is unique. However, this need not be the case in modular complemented lattices. For example, consider the lattice $\mathbf M_n=(M_n,\vee,\wedge,0,1)$ (for $n>1$) depicted in Figure~1:

\vspace*{-3mm}

\begin{center}
	\setlength{\unitlength}{7mm}
	\begin{picture}(8,6)
		\put(4,1){\circle*{.3}}
		\put(1,3){\circle*{.3}}
		\put(3,3){\circle*{.3}}
		\put(7,3){\circle*{.3}}
		\put(4,5){\circle*{.3}}
		\put(4,1){\line(-3,2)3}
		\put(4,1){\line(-1,2)1}
		\put(4,1){\line(3,2)3}
		\put(4,5){\line(-3,-2)3}
		\put(4,5){\line(-1,-2)1}
		\put(4,5){\line(3,-2)3}
		\put(3.85,.3){$0$}
		\put(.25,2.85){$a_1$}
		\put(2.25,2.85){$a_2$}
		\put(4.5,2.85){$\cdots$}
		\put(7.3,2.85){$a_n$}
		\put(3.85,5.4){$1$}
		\put(3.2,-.75){{\rm Fig.~1}}
		\put(2,-1.75){The lattice $\mathbf M_n$}
	\end{picture}
\end{center}

\vspace*{10mm}

Then for every $i,j\in\{1,\ldots,n\}$ with $i\ne j$, the element $a_j$ is a complement of $a_i$.

Sometimes, for lattices with complementation, we ask if this complementation is {\em antitone}, i.e.\ if $x\le y$ implies $y'\le x'$, or if it is an {\em involution}, i.e.\ $x''=x$. In distributive complemented lattices the complementation turns out to be unique, antitone and an involution. In such a case the lattice is a Boolean algebra.

Within modular lattices the situation may be different. Consider the complemented modular lattice $\mathbf L=(L,\vee,\wedge,0,1)$ visualized in Figure~2:

\vspace*{-3mm}

\begin{center}
	\setlength{\unitlength}{7mm}
	\begin{picture}(14,8)
		\put(4,1){\circle*{.3}}
		\put(1,3){\circle*{.3}}
		\put(3,3){\circle*{.3}}
		\put(5,3){\circle*{.3}}
		\put(7,3){\circle*{.3}}
		\put(10,3){\circle*{.3}}
		\put(4,5){\circle*{.3}}
		\put(7,5){\circle*{.3}}
		\put(9,5){\circle*{.3}}
		\put(11,5){\circle*{.3}}
		\put(13,5){\circle*{.3}}
		\put(10,7){\circle*{.3}}
		\put(4,1){\line(-3,2)3}
		\put(4,1){\line(-1,2)1}
		\put(4,1){\line(1,2)1}
		\put(4,1){\line(3,2)3}
		\put(10,3){\line(-3,2)3}
		\put(10,3){\line(-1,2)1}
		\put(10,3){\line(1,2)1}
		\put(10,3){\line(3,2)3}
		\put(4,5){\line(-3,-2)3}
		\put(4,5){\line(-1,-2)1}
		\put(4,5){\line(1,-2)1}
		\put(4,5){\line(3,-2)3}
		\put(10,7){\line(-3,-2)3}
		\put(10,7){\line(-1,-2)1}
		\put(10,7){\line(1,-2)1}
		\put(10,7){\line(3,-2)3}
		\put(4,1){\line(3,1)6}
		\put(1,3){\line(3,1)6}
		\put(3,3){\line(3,1)6}
		\put(5,3){\line(3,1)6}
		\put(7,3){\line(3,1)6}
		\put(4,5){\line(3,1)6}
		\put(3.85,.3){$0$}
		\put(.4,2.85){$a$}
		\put(2.4,2.85){$b$}
		\put(5.3,2.85){$c$}
		\put(7.3,2.85){$d$}
		\put(10.3,2.85){$e$}
		\put(3.4,4.85){$f$}
		\put(6.4,4.85){$g$}
		\put(8.4,4.85){$h$}
		\put(11.3,4.85){$i$}
		\put(13.3,4.85){$j$}
		\put(9.85,7.4){$1$}
		\put(6.2,-.75){{\rm Fig.~2}}
		\put(3,-1.75){Complemented modular lattice}
	\end{picture}
\end{center}

\vspace*{10mm}

Evidently; $\mathbf L$ is a complemented lattice. We have several choices for defining a complementation $'$. If we define $'$ by
\[
\begin{array}{l|cccccccccccc}
	x  & 0 & a & b & c & d & e & f & g & h & i & j & 1 \\
	\hline
	x' & 1 & h & i & j & g & f & e & b & c & d & a & 0
\end{array}
\]
then it is not an involution. If we define $'$ by
\[
\begin{array}{l|cccccccccccc}
	x  & 0 & a & b & c & d & e & f & g & h & i & j & 1 \\
	\hline
	x' & 1 & h & i & j & g & f & e & d & a & b & c & 0
\end{array}
\]
then it is an antitone involution and hence $\mathbf L=(L,\vee,\wedge,{}',0,1)$ is a so-called {\em orthomodular lattice} (see e.g.\ \cite{Be} for the definition).

Hence, not every modular lattice endowed with a complementation must be orthomodular. Of course, not every orthomodular lattice is modular (see \cite{Be}).

If $\mathbf L=(L,\vee,\wedge,0,1)$ is a complemented lattice in which the complementation is not introduced in form of a unary operation then we need not distinguish between the complements of a given element $a$ of $L$. Hence we will work with the whole set of complements of $a$. Within this paper we will use this approach.

We start by introducing some lattice-theoretical concepts.

All complemented lattices considered within this paper are assumed to be non-trivial, i.e.\ to have a bottom element $0$ and a top element $1$ with $0\ne1$.

Let $(L,\vee,\wedge,0,1)$ be a complemented lattice and $A,B\subseteq L$. We define:
\begin{align*}
  A\vee B & :=\{x\vee y\mid x\in A\text{ and }y\in B\}, \\
A\wedge B & :=\{x\wedge y\mid x\in A\text{ and }y\in B\}, \\
   A\le B & \text{ if }x\le y\text{ for all }x\in A\text{ and all }y\in B, \\
  A\le_1B & \text{ if for every }x\in A\text{ there exists some }y\in B\text{ with }x\le y, \\
  A\le_2B & \text{ if for every }y\in B\text{ there exists some }x\in A\text{ with }x\le y.
\end{align*}

\section{The operator $^+$}
	
Let $\mathbf L=(L,\vee,\wedge,0,1)$ be a complemented lattice. For $a\in L$ we define
\[
a^+:=\{x\in L\mid a\vee x=1\text{ and }a\wedge x=0\},
\]
i.e.\ $a^+$ is the set of all complements of $a$. Since $\mathbf L$ is complemented, we have $a^+\ne\emptyset$ for all $a\in L$. For every subset $A$ of $L$ we put
\[
A^+:=\{x\in L\mid a\vee x=1\text{ and }a\wedge x=0\text{ for all }a\in A\}.
\]
Observe that $A^+$ may me empty, e.g.\ $L^+=\emptyset$ (and $\emptyset^+=L$). In the following we often identify singletons with their unique element.

\begin{example}\label{ex1}
	For the lattice $\mathbf N_5$ depicted in Figure~3:

\vspace*{-3mm}

\begin{center}
	\setlength{\unitlength}{7mm}
	\begin{picture}(4,8)
		\put(2,1){\circle*{.3}}
		\put(1,4){\circle*{.3}}
		\put(3,3){\circle*{.3}}
		\put(3,5){\circle*{.3}}
		\put(2,7){\circle*{.3}}
		\put(2,1){\line(-1,3)1}
		\put(2,1){\line(1,2)1}
		\put(3,3){\line(0,1)2}
		\put(2,7){\line(-1,-3)1}
		\put(2,7){\line(1,-2)1}
		\put(1.85,.3){$0$}
		\put(3.4,2.85){$a$}
		\put(.35,3.85){$b$}
		\put(3.4,4.85){$c$}
		\put(1.85,7.4){$1$}
		\put(1.2,-.75){{\rm Fig.~3}}
		\put(-1.1,-1.75){Non-modular lattice $\mathbf N_5$}
	\end{picture}
\end{center}

\vspace*{8mm}

we have
\[
\begin{array}{l|r|r|r|r|r}
	x      & 0 &  a &  b &  c & 1 \\
	\hline
    x^+    & 1 &  b & ac &  b & 0 \\
	\hline
	x^{++} & 0 & ac &  b & ac & 1
\end{array}
\]
Here and in the following within tables we sometimes write $abc$ instead of $\{a,b,c\}$.
	For the lattice $\mathbf M_3$ visualized in Figure~4

\vspace*{-3mm}

\begin{center}
	\setlength{\unitlength}{7mm}
	\begin{picture}(6,6)
		\put(3,1){\circle*{.3}}
		\put(1,3){\circle*{.3}}
		\put(3,3){\circle*{.3}}
		\put(5,3){\circle*{.3}}
		\put(3,5){\circle*{.3}}
		\put(3,1){\line(-1,1)2}
		\put(3,1){\line(0,1)4}
		\put(3,1){\line(1,1)2}
		\put(3,5){\line(-1,-1)2}
		\put(3,5){\line(1,-1)2}
		\put(2.85,.3){$0$}
		\put(.35,2.85){$a$}
		\put(3.4,2.85){$b$}
		\put(5.4,2.85){$c$}
		\put(2.85,5.4){$1$}
		\put(2.2,-.75){{\rm Fig.~4}}
		\put(.5,-1.75){Modular lattice $\mathbf M_3$}
	\end{picture}
\end{center}

\vspace*{8mm}

we have
\[
\begin{array}{r|r|r|r|r|r}
	x      & 0 &  a &  b &  c & 1 \\
	\hline
	x^+    & 1 & bc & ac & ab & 0 \\
	\hline
	x^{++} & 0 &  a &  b &  c & 1
\end{array}
\]
Let us note that $\mathbf M_3$ satisfies the identity $x^{++}\approx x$.
\end{example}

\begin{example}
	For the example from Figure~2 we have
	\[
	\begin{array}{l|r|r|r|r|r|r|r|r|r|r|r|r}
		x      & 0 &   a &   b &   c &   d & e & f &   g &   h &   i &   j & 1 \\
		\hline
		x^+    & 1 & hij & gij & ghj & ghi & f & e & bcd & acd & abd & abc & 0 \\
		\hline
		x^{++} & 0 &   a &   b &   c &   d & e & f &   g &   h &   i &   j & 1
	\end{array}
	\]
\end{example}

Recall the concept of a {\em Galois connection} which is often used in lattices. The pair $(^+,^+)$ is the Galois connection between $(2^L,\subseteq)$ and $(2^L,\subseteq)$ induced by the relation
\[
\{(x,y)\in L^2\mid x\vee y=1\text{ and }x\wedge y=0\}.
\]
From this we conclude
\begin{align*}
	             A & \subseteq A^{++}, \\
	  A\subseteq B & \Rightarrow B^+\subseteq A^+, \\
	       A^{+++} & =A^+, \\
	A\subseteq B^+ & \Leftrightarrow B\subseteq A^+
\end{align*}
for all $A,B\subseteq L$. Since $A\subseteq A^{++}$ we have that $A^{++}\ne\emptyset$ whenever $A\ne\emptyset$. A subset $A$ of $L$ is called {\em closed} if $A^{++}=A$. Let $\Cl(\mathbf L)$ denote the set of all closed subsets of $L$. Then clearly $\Cl(\mathbf L)=\{A^+\mid A\subseteq L\}$. Because of $A^+\cap A^{++}=\emptyset$ for all $A\subseteq L$ we have that $\big(\Cl(\mathbf L),\subseteq,{}^+,\emptyset,L\big)$ forms a complete ortholattice with
\begin{align*}
	\bigvee_{i\in I}A_i & =\left(\bigcup_{i\in I}A_i\right)^{++}, \\
	\bigwedge_{i\in I}A_i & =\bigcap_{i\in I}A_i
\end{align*}
for all families $(A_i;i\in I)$ of closed subsets of $L$.

Next we describe the basic properties of the operator $^+$.

\begin{proposition}\label{prop1}
	Let $\mathbf L=(L,\vee,\wedge,0,1)$ be a complemented lattice and $a\in L$. Then the following holds:
	\begin{enumerate}[{\rm(i)}]
		\item $a\in a^{++}$ and $a^{+++}=a^+$,
		\item $(x^+,\le)$ is an antichain for every $x\in L$ if and only if $\mathbf L$ does not contain a sublattice isomorphic to $\mathbf N_5$ containing $0$ and $1$,
		\item $(a^+,\le)$ is convex,
		\item if the mapping $x\mapsto x^{++}$ from $L$ to $2^L$ is not injective then $\mathbf L$ does not satisfy the identity $x^{++}\approx x$.
	\end{enumerate}
\end{proposition}

\begin{proof}
	\
	\begin{enumerate}[(i)]
		\item follows directly from above.
		\item First assume there exists some $b\in L$ such that $(b^+,\le)$ is not an antichain. Then $b\notin\{0,1\}$. Now there exist $c,d\in b^+$ with $c<d$. Since $b\notin\{0,1\}$ and $b\in c^+\cap d^+$ we have $c,d\notin\{0,1\}$. Because of $|L|>1$ we have $b\notin\{c,d\}$. Hence the elements $0$, $b$, $c$, $d$ and $1$ are pairwise distinct and form an $\mathbf N_5$ containing $0$ and $1$. If, conversely, $\mathbf L$ contains a sublattice $(L_1,\vee,\wedge)$ isomorphic to $\mathbf N_5$ and containing $0$ and $1$, say $L_1=\{0,e,f,g,1\}$ with $e<f$ then $e,f\in g^+$ and hence $(g^+,\le)$ is not an antichain.
		\item If $b,c\in a^+$, $d\in L$ and $b\le d\le c$ then $1=a\vee b\le a\vee d$ and $a\wedge d\le a\wedge c=0$ showing $d\in a^+$.
		\item If the mapping $x\mapsto x^{++}$ is not injective then there exist $a,b\in L$ with $a\ne b$ and $a^{++}=b^{++}$ which implies $b\in b^{++}=a^{++}$ and $b\ne a$ and hence $a^{++}\ne a$ showing that $\mathbf L$ does not satisfy the identity $x^{++}\approx x$.
	\end{enumerate}
\end{proof}

In the lattice $\mathbf N_5$ from Example~\ref{ex1} the mapping $x\mapsto x^{++}$ is not injective since $a\ne c$ and $a^{++}=c^{++}$. According to Proposition~\ref{prop1} (iv), this lattice does not satisfy the identity $x^{++}\approx x$, e.g.\ $a^{++}=\{a,c\}\ne a$.

\begin{corollary}\label{cor1}
	Let $(L,\vee,\wedge,0,1)$ be a complemented modular lattice, $a\in L$ and $A$ a non-empty subset of $L$. According to Proposition~\ref{prop1} {\rm(iii)}, $(a^+,\le)$ is an antichain. Let $b\in A$. Then $A^+\subseteq b^+$ and hence also $(A^+,\le)$ is an antichain. Since $a^+$ is a non-empty subset of $L$ we finally conclude that $(a^{++},\le)$ is an antichain, too.
\end{corollary}

In case of finite $L$ we can even prove the following.

\begin{proposition}
	Let $(L,\vee,\wedge,0,1)$ be a finite complemented lattice such that $x\mapsto x^{++}$ is injective and $a\in L$ and assume $a^{++}\ne a$. Then there exists some $b\in a^{++}$ with $b^{++}=b$.
\end{proposition}

\begin{proof}
	Let $a_1\in a^{++}\setminus\{a\}$. Then $a_1^{++}\subseteq a^{++}$. Since $a_1\ne a$ and $x\mapsto x^{++}$ is injective we conclude $a_1^{++}\subsetneqq a^{++}$. Now either $a_1^{++}=a_1$ or there exists some $a_2\in a_1^{++}\setminus\{a_1\}$. Then $a_2^{++}\subseteq a_1^{++}$. Since $a_2\ne a_1$ and $x\mapsto x^{++}$ is injective we conclude $a_2^{++}\subsetneqq a_1^{++}$. Now either $a_2^{++}=a_2$ or there exists some $a_3\in a_2^{++}\setminus\{a_2\}$. Since $L$ is finite and $a_1^{++}\supsetneqq a_2^{++}\supsetneqq\cdots$ there exists some $n\ge1$ with $|a_n^{++}|=1$, i.e.\ $a_n^{++}=a_n$ and we have $a_n\in a_n^{++}\subseteq a_{n-1}^{++}\subseteq\cdots\subseteq a_1^{++}\subseteq a^{++}$.
\end{proof}

The relationship between the operator $^+$ and the partial order relation of $\mathbf L$ is illuminated in the following result.

\begin{proposition}
	Let $(L,\vee,\wedge,0,1)$ be a complemented lattice and consider the following statements:
	\begin{enumerate}[{\rm(i)}]
		\item $x^+\vee y^+\le_1(x\wedge y)^+$ for all $x,y\in L$,
		\item for all $x,y\in L$, $x\le y$ implies $y^+\le_1x^+$,
		\item $(x\vee y)^+\le_1x^+\wedge y^+$ for all $x,y\in L$.
	\end{enumerate}
	Then {\rm(i)} $\Rightarrow$ {\rm(ii)} $\Leftrightarrow$ {\rm(iii)}.
\end{proposition}

\begin{proof}
	Let $a,b\in L$. \\
	(i) $\Rightarrow$ (ii): \\
	$a\le b$ implies $b^+=(a\vee b)^+\le_1a^+\wedge b^+\le_1a^+$. \\
	(ii) $\Rightarrow$ (iii): \\
	Because of $a,b\le a\vee b$ we have $(a\vee b)^+\le_1a^+,b^+$ which implies $(a\vee b)^+\le_1a^+\wedge b^+$. \\
	(iii) $\Rightarrow$ (ii): \\
	$a\le b$ implies $b^+\le_1a^+\vee b^+\le(a\wedge b)^+=a^+$.
\end{proof}

Our next task is to characterize the property that a complemented lattice $\mathbf L$ satisfies the identity $x^{++}\approx x$. From Example~\ref{ex1} we know that if $\mathbf L$ is not modular then this identity need not hold. Hence we restrict ourselves to complemented modular lattices.

\begin{theorem}
	Let $\mathbf L=(L,\vee,\wedge,0,1)$ be a complemented modular lattice. Then the following are equivalent:
	\begin{enumerate}[{\rm(i)}]
		\item $\mathbf L$ satisfies the identity $x^{++}\approx x$,
		\item for every $x\in L$ and each $y\in x^{++}$ there exists some $z\in y^+$ satisfying either $(x\vee y)\wedge z=0$ or $(x\wedge y)\vee z=1$.
	\end{enumerate}
\end{theorem}

\begin{proof}
	$\text{}$ \\
	(i) $\Rightarrow$ (ii): \\
	If $a\in L$, $b\in a^{++}$ and $c\in b^+$ then $b=a$ and $(a\vee b)\wedge c=b\wedge c=0$. \\
	(ii) $\Rightarrow$ (i): \\
	Assume (ii). Suppose $\mathbf L$ not to satisfy the identity $x^{++}\approx x$. Then there exists some $a\in L$ with $a^{++}\ne a$. Let $b\in a^{++}\setminus\{a\}$. According to (ii) there exists some $c\in b^+$ satisfying either $(a\vee b)\wedge c=0$ or $(a\wedge b)\vee c=1$. Since $a$ and $b$ are different elements of $a^{++}$ and $(a^{++},\le)$ is an antichain according to Corollary~\ref{cor1}, we conclude $a\parallel b$. Now $(a\vee b)\wedge c=0$ would imply
	\[
	a\le a\vee b=1\wedge(a\vee b)=(b\vee c)\wedge(a\vee b)=b\vee\big(c\wedge(a\vee b)\big)=b\vee0=b,
	\]
	contradicting $a\parallel b$. On the other hand, $(a\wedge b)\vee c=1$ would imply
	\[
	b=1\wedge b=\big((a\wedge b)\vee c\big)\wedge b=(a\wedge b)\vee(c\wedge b)=(a\wedge b)\vee0=a\wedge b\le a,
	\]
	again contradicting $a\parallel b$. This shows that $\mathbf L$ satisfies the identity $x^{++}\approx x$.
\end{proof}

\section{The operator $\to$}

Let $\mathbf L=(L,\vee,\wedge,{}',0,1)$ be an orthomodular lattice. Recall that the operation $\phi_x$ defined by $\phi_x(y):=x\wedge(x'\vee y)$ for all $x,y\in L$ was introduced by U.~Sasaki in \cite{Sa50} and \cite{Sa52} and is called the {\em Sasaki projection} (see e.g.\ \cite{Be}) or {\em Sasaki hook} alias {\em Sasaki operation}, see \cite{CL}; its dual, i.e.\ the operation $\psi_x$ defined by $\psi_x(y):=x'\vee(x\wedge y)$ for all $x,y\in L$ is then called the {\em dual Sasaki projection}. It was shown by the authors in \cite{CL} that if we use these Sasaki operations in order to define
\begin{align*}
  x\to y & :=x'\vee(x\wedge y), \\
x\cdot y & :=(x\vee y')\wedge y
\end{align*}
for all $x,y$ belonging to the base set of the orthomodular lattice $\mathbf L$ then the operations $\to$ and $\cdot$ form an adjoint pair, i.e.\
\[
x\cdot y\le z\text{ if and only if }x\le y\to z
\]
for all $x,y,z\in L$. This motivated us to introduce our next operators in a similar way where, however, instead of the element $x'$ we use the set $x^+$. Hence, for a complemented lattice $(L,\vee,\wedge,0,1)$, $a,b\in L$ and $A,B\subseteq L$ we define
\begin{align*}
a\to b & :=a^+\vee(a\wedge b), \\
A\to B & :=A^+\vee(A\wedge B).
\end{align*}
Observe that $A\to B=\emptyset$ whenever $A^+=\emptyset$.

\begin{example}
	For the lattice from Figure~2 we have e.g.
	\begin{align*}
		a\to b & =\{h,i,j\}\vee(a\wedge b)=\{h,i,j\}\vee0=\{h,i,j\}=a^+, \\
		a\to f & =\{h,i,j\}\vee(a\wedge f)=\{h,i,j\}\vee a=1, \\
		a\to g & =\{h,i,j\}\vee(a\wedge g)=\{h,i,j\}\vee a=1, \\
		a\to h & =\{h,i,j\}\vee(a\wedge h)=\{h,i,j\}\vee0=\{h,i,j\}=a^+, \\
		f\to e & =e\vee(f\wedge e)=e\vee0=e, \\
		g\to h & =\{b,c,d\}\vee(g\wedge h)=\{b,c,d\}\vee e=\{h,i,j\}=a^+.
	\end{align*}
\end{example}

In the following we study the relationship between $\to$ and $\wedge$.

\begin{theorem}
	Let $(L,\vee,\wedge,0,1)$ be a complemented modular lattice and $a,b,c\in L$. Then the following holds:
	\begin{enumerate}[{\rm(i)}]
		\item If $a\le_1b\to c$ then $a\wedge b\le c$,
		\item $a\wedge b\le c$ if and only if $a\wedge b\le_1b\to c$.
	\end{enumerate}
\end{theorem}

\begin{proof}
	\
	\begin{enumerate}[(i)]
		\item
	From $a\le_1b\to c$ we conclude that there exists some $d\in b^+$ satisfying $a\le d\vee(b\wedge c)$, and we obtain
\[
a\wedge b\le\big(d\vee(b\wedge c)\big)\wedge b=\big((b\wedge c)\vee d\big)\wedge b=(b\wedge c)\vee(d\wedge b)=(b\wedge c)\vee0=b\wedge c\le c.
\]
		\item
	First assume $a\wedge b\le c$. Let $e\in b^+$. Then
	\[
	a\wedge b\le e\vee(a\wedge b)=e\vee\big(b\wedge(a\wedge b)\big)\le e\vee(b\wedge c).
	\]
	This shows $a\wedge b\le_1b\to c$. Conversely, assume $a\wedge b\le_1b\to c$. Then there exists some $f\in b^+$ with $a\wedge b\le f\vee(b\wedge c)$. So we get
	\[
	a\wedge b=(a\wedge b)\wedge b\le\big(f\vee(b\wedge c)\big)\wedge b=\big((b\wedge c)\vee f\big)\wedge b=(b\wedge c)\vee(f\wedge b)=(b\wedge c)\vee0=b\wedge c\le c.
	\]
	\end{enumerate}
\end{proof}

For complemented lattices, the operator $\to$ satisfies a lot of properties common in residuated structures.

\begin{theorem}\label{th1}
Let $(L,\vee,\wedge,0,1)$ be a complemented lattice and $a,b,c\in L$. Then the following holds:
\begin{enumerate}[{\rm(i)}]
	\item $a\to0=a^+$ and $1\to a=a$,
	\item If $a\le b$ then $a\to b=1$,
	\item $a\to b=1$ if and only if $a\wedge b\in a^{++}$,
	\item if $b\in a^+$ then $a\to b=a^+$,
	\item if $b\le c$ then $a\to b\le_ia\to c$ for $i=1,2$,
	\item if $a\to b=a\to c=1$ and $a^{++}$ is closed with respect to $\wedge$ then $a\to(b\wedge c)=1$,
	\item if $a^{++}\subseteq b^{++}$ and $a\to b=1$ then $b\to a=1$.
\end{enumerate}
\end{theorem}

\begin{proof}
	\
	\begin{enumerate}[(i)]
		\item We have $a\to0=a^+\vee(a\wedge0)=a^+\vee0=a^+$ and $1\to a=1^+\vee(1\wedge a)=0\vee a=a$.
		\item If $a\le b$ then $a\to b=a^+\vee(a\wedge b)=a^+\vee a=1$.
		$d=e\vee(a\wedge b)$. Now we have $d=e\vee(a\wedge b)\le e\vee(a\wedge c)\in a\to c$.
		\item The following are equivalent:
		\begin{align*}
			a\to b & =1, \\
			a^+\vee(a\wedge b) & =1, \\
			a^+\vee(a\wedge b) & =1\text{ and }a^+\wedge(a\wedge b)=0, \\
			a\wedge b & \in a^{++}.
		\end{align*}
		\item If $b\in a^+$ then $a\to b=a^+\vee(a\wedge b)=a^+\vee0=a^+$.
		\item If $b\le c$ then $a\to b=a^+\vee(a\wedge b)\le_ia^+\vee(a\wedge c)=a\to c$ for $i=1,2$.
		\item Using (v) and the assumptions we obtain $a\wedge b,a\wedge c\in a^{++}$ and hence $a\wedge(b\wedge c)=(a\wedge b)\wedge(a\wedge c)\in a^{++}$ showing $a\to(b\wedge c)=1$.
		\item Using (v) and the assumptions we have $a\to b=1$ and hence $b\wedge a=a\wedge b\in a^{++}\subseteq b^{++}$ which implies $b\to a=1$.
	\end{enumerate}
\end{proof}

Let us note that the converse of Theorem~\ref{th1} (ii) does not hold in general. For example, consider the lattice $\mathbf N_5$ from Example~\ref{ex1}. Then $c\to a=c^+\vee(c\wedge a)=b\vee a=1$ contrary to the fact that $c>a$. However, if $x$ is a minimal element of $x^{++}$ then we can prove the following.

\begin{proposition}
Let $(L,\vee,\wedge,0,1)$ be a complemented lattice and $a\in L$. Then the following are equivalent:
	\begin{enumerate}[{\rm(i)}]
	\item For all $x\in L$, $a\to x=1$ is equivalent to $a\le x$,
	\item $a$ is a minimal element of $a^{++}$.
	\end{enumerate}
\end{proposition}

\begin{proof}
According to (ii) of Theorem~\ref{th1} the following are equivalent:
\begin{align*}
	& \text{For all }x\in L, a\to x=1\text{ is equivalent to }a\le x, \\
	& \text{for all }x\in L, a\wedge x\in a^{++}\text{ is equivalent to }a\wedge x=a, \\
	& \text{for all }y\le a, y\in a^{++}\text{ is equivalent to }y=a, \\
	& a\text{ is a minimal element of }a^{++}.
\end{align*}
\end{proof}

We are  going to show how the operator $\to$ is related to the connective implication in a propositional calculus.

\begin{theorem}\label{th2}
	Let $\mathbf L=(L,\vee,\wedge,0,1)$ be a complemented modular lattice and $a,b\in L$. Then the following holds:
	\begin{enumerate}[{\rm(i)}]
		\item $a\wedge(a\to b)=a\wedge b\le b$ {\rm(}{\em Modus Ponens}{\rm)},
		\item if $a^+\le b^+$ then $(a\to b)\wedge b^+=a^+$ {\rm(}{\em Modus Tollens}{\rm)},
		\item if $c\in a\to b$ then $a\to c=a\to b$,
		\item $a\to(a\to b)=a\to b$,
		\item if $a^+\le b$ then $a\to b=b$,
	\end{enumerate}
\end{theorem}

\begin{proof}
	\
	\begin{enumerate}[(i)]
		\item Using modularity of $\mathbf L$ we compute
\begin{align*}
	a\wedge(a\to b) & =a\wedge\big(a^+\vee(a\wedge b)\big)=\big((a\wedge b)\vee a^+\big)\wedge a=(a\wedge b)\vee(a^+\wedge a)= \\
	& =(a\wedge b)\vee0=a\wedge b\le b.
\end{align*}
\item Under the assumptions
\[
(a\to b)\wedge b^+=\big(a^+\vee(a\wedge b)\big)\wedge b^+=a^+\vee\big((a\wedge b)\wedge b^+\big)=a^+\vee0=a^+.
\]
\item If $c\in a\to b$ then there exists some $d\in a^+$ with $d\vee(a\wedge b)=c$ and hence
\begin{align*}
	a\to c & =a^+\vee\Big(a\wedge\big(d\vee(a\wedge b)\big)\Big)=a^+\vee\Big(\big((a\wedge b)\vee d\big)\wedge a\Big)= \\
	       & =a^+\vee\big((a\wedge b)\vee(d\wedge a)\big)=a^+\vee\big((a\wedge b)\vee0\big)=a^+\vee(a\wedge b)=a\to b.
\end{align*}
\item Using (iii) we obtain
\begin{align*}
a\to(a\to b) & =a^+\vee\big(a\wedge(a\to b)\big)=\bigcup_{c\in a\to b}\big(a^+\vee(a\wedge c)\big)=\bigcup_{c\in a\to b}(a\to c)= \\
             & =\bigcup_{c\in a\to b}(a\to b)=a\to b.
\end{align*}
\item If $a^+\le b$ then $a\to b=a^+\vee(a\wedge b)=(a^+\vee a)\wedge b=1\wedge b=b$.
\end{enumerate}
\end{proof}

\begin{proposition}
	Let $n>1$ and $a,b,c\in M_n$. Then
	\[
		a\wedge b\le c\text{ if and only if }a\le_1b\to c.
	\]
\end{proposition}

\begin{proof}
	It is easy to see that
	\[
	a\to b=\left\{
	\begin{array}{ll}
		1 & a\le b, \\
		b & a=1, \\
		a^+ & a\parallel b\text{ or }b=0
	\end{array}
	\right.
	\]
	If $a=0$ then $a\wedge b=0\wedge b=0\le c$ and $a=0\le_1b\to c$. If $b\le c$ then $a\wedge b\le b\le c$ and $a\le_11=b\to c$. If $b=1$ then both (i) and (ii) are equivalent to $a\le c$. Hence we can assume $a\ne0$, $b\not\le c$ and $b\ne1$. In case $n>2$ let $a$, $b$ and $c$ be pairwise different elements of $M_n\setminus\{0,1\}$ and in case $n=2$ let $M_n=\{0,a,b,1\}$. Then the following cases remain:
	\[
	\begin{array}{r|r|r|r|r}
		x & y & z & x\wedge y\le z & x\le_1y\to z \\
		\hline
		a & a & 0 &      \text{no} &    \text{no} \\
		\hline
		a & a & b &      \text{no} &    \text{no} \\
		\hline
		a & b & 0 &     \text{yes} &   \text{yes} \\
		\hline
		a & b & a &     \text{yes} &   \text{yes} \\ 
		\hline
		a & b & c &     \text{yes} &   \text{yes} \\
		\hline
		1 & a & 0 &      \text{no} &    \text{no} \\
		\hline
		1 & a & b &      \text{no} &    \text{no}
	\end{array}
	\]
\end{proof}

\section{The operator $\odot$}

Similarly as it was done in Section~3 concerning the operator $\to$, also here we define the new operator $\odot$ by means of the generalized Sasaki projection.

For a complemented lattice $(L,\vee,\wedge,0,1)$, $a,b\in L$ and $A,B\subseteq L$ we define
\begin{align*}
a\odot b & :=b\wedge(a\vee b^+), \\
A\odot B & :=B\wedge(A\vee B^+).
\end{align*}
It is evident that $\odot$ need neither be commutative nor associative, but it is idempotent, i.e.\ it satisfies the identity $x\odot x\approx x$ (cf.\ Proposition~\ref{prop2} (iii)).

We list some basic properties of the operator $\odot$.

\begin{proposition}\label{prop2}
	Let $\mathbf L=(L,\vee,\wedge,0,1)$ a complemented lattice and $a,b,c\in L$. Then the following holds:
	\begin{enumerate}[{\rm(i)}]
		\item $0\odot a=a\odot0=0$,
		\item $1\odot a=a\odot1=a$,
		\item $a\wedge b\le a\odot b\le b$ and if $b\le a$ then $a\odot b=b$,
		\item if $a\le b$ then $a\odot c\le_ib\odot c$ for $i=1,2$,
		\item if $\mathbf L$ is modular then $a\le b$ if and only if $a\odot b=a$ and, moreover, $(a\odot b)\odot b=a\odot b$.
		\end{enumerate}
\end{proposition}

\begin{proof}
	\
	\begin{enumerate}[(i)]
		\item We have $0\odot a=a\wedge(0\vee a^+)=a\wedge a^+=0$ and $a\odot0=0\wedge(a\vee0^+)=0$.
		\item We have $1\odot a=a\wedge(1\vee a^+)=a\wedge1=a$ and $a\odot1=1\wedge(a\vee1^+)=a\vee0=a$.
		\item This follows from the definition of $a\odot b$.
		\item If $a\le b$ then $a\odot c=c\wedge(a\vee c^+)\le_ic\wedge(b\vee c^+)=b\odot c$ for $i=1,2$.
		\item If $a\le b$ then using modularity of $\mathbf L$ we obtain
		\[
		a\odot b=b\wedge(a\vee b^+)=(a\vee b^+)\wedge b=a\vee(b^+\wedge b)=a\vee0=a.
		\]
		That $a\odot b=a$ implies $a\le b$ follows from (iii). Using (iii) and modularity of $\mathbf L$ we obtain
		\[
		(a\odot b)\odot b=b\wedge\big((a\odot b)\vee b^+\big)=\big((a\odot b)\vee b^+\big)\wedge b=(a\odot b)\vee(b^+\wedge b)=(a\odot b)\vee0=a\odot b.
		\]
	\end{enumerate}
\end{proof}

\begin{example}
	The ``operation tables'' for $\odot$ for the lattices $\mathbf N_5$ and $\mathbf M_3$ {\rm(}see Example~\ref{ex1}{\rm)} are as follows:
	\[
	\begin{array}{c|ccccc}
		\odot & 0 & a & b & c & 1 \\
		\hline
		   0  & 0 & 0 & 0 & 0 & 0 \\
		   a  & 0 & a & 0 & c & a \\
		   b  & 0 & 0 & b & 0 & b \\
		   c  & 0 & a & 0 & c & c \\
		   1  & 0 & a & b & c & 1
	\end{array}
\quad\quad
\begin{array}{r|rrrrr}
	\odot & 0 & a & b & c & 1 \\
	\hline
	   0  & 0 &  0 &  0 &  0 & 0 \\
	   a  & 0 &  a & 0b & 0c & a \\
	   b  & 0 & 0a &  b & 0c & b \\
	   c  & 0 & 0a & 0b &  c & c \\
	   1  & 0 &  a &  b &  c & 1
\end{array}
\]
\hspace*{54mm}$\mathbf N_5$\hspace*{40mm}$\mathbf M_3$
\end{example}

Contrary to the relatively week relationship between $\to$ and $\wedge$, for $\odot$ and $\to$ we can prove here a kind of adjointness.

\begin{theorem}
	Let $(L,\vee,\wedge,0,1)$ be a complemented modular lattice and $a,b,c\in L$. Then
	\[
		a\odot b\le c\text{ if and only if }a\le b\to c.
	\]
\end{theorem}

\begin{proof}
	If $a\odot b\le c$ then $b\wedge(a\vee x)\le c$ for all $x\in b^+$ and hence
	\[
	a\le a\vee x=1\wedge(a\vee x)=(x\vee b)\wedge(a\vee x)=x\vee\big(b\wedge(a\vee x)\big)=x\vee\Big(b\wedge\big(b\wedge(a\vee x)\big)\Big)\le x\vee(b\wedge c)
	\]
	for all $x\in b^+$ showing $a\le b\to c$. If, conversely, $a\le b\to c$ then $a\le x\vee(b\wedge c)$ for all $x\in b^+$ and hence
	\begin{align*}
		b\wedge(a\vee x) & \le b\wedge\Big(\big(x\vee(b\wedge c)\big)\vee x\Big)=b\wedge\big(x\vee(b\wedge c)\big)=\big((b\wedge c)\vee x\big)\wedge b= \\
		& =(b\wedge c)\vee(x\wedge b)=(b\wedge c)\vee0=b\wedge c\le c
	\end{align*}
	for all $x\in b^+$ showing $a\odot b\le c$.
\end{proof}

\section{Deductive systems}

Deductive systems are often introduced in algebras forming an algebraic formalization of a non-classical propositional calculus. These are subsets of the algebra in question containing the logical constant $1$ and representing the derivation rule Modus Ponens. Since our operator $\to$ shares a number of properties with the non-classical logical connective implication, we define this concept also for complemented lattices.

\begin{definition}
	A {\em deductive system} of a complemented lattice $\mathbf L=(L,\vee,\wedge,0,1)$ is a subset $D$ of $L$ satisfying the following conditions:
	\begin{itemize}
		\item $1\in D$,
		\item if $a\in D$, $b\in L$ and $a\to b\subseteq D$ then $b\in D$.
	\end{itemize}
\end{definition}

Since the intersection of deductive systems of $\mathbf L$ is again a deductive system of $\mathbf L$, the set of all deductive systems of $\mathbf L$ forms a complete lattice $\BDed\mathbf L$ with respect to inclusion with bottom element $\{1\}$ and top element $L$.

\begin{example}\label{ex2}
	The deductive systems of the lattice $\mathbf M_n$ for $n>1$ {\rm(}see Figure~1{\rm)} are given by $M_n$ and $A\cup\{1\}$ where $A$ is a proper subset of $\{a_1,\ldots,a_n\}$. This can be seen as follows. Let $i,j\in\{1,\ldots,n\}$ with $i\ne j$. Then we have
	\[
	\begin{array}{l|llll}
		\to & 0 & a_i & a_j & 1 \\
		\hline
		0   & 1     & 1     & 1     & 1 \\
		a_i & a_i^+ & 1     & a_i^+ & 1 \\
		a_j & a_j^+ & a_j^+ & 1     & 1 \\
		1   & 0     & a_i   & a_j   & 1
	\end{array}
	\]
Now let $D$ be a deductive system of $\mathbf M_n$. Then the following hold: \\
$1\in D$, \\
if $0\in D$ then $D=M_n$, \\
if $\{a_1,\ldots,a_n\}\subseteq D$ then $a_1\in D$ and $a_1\to0=a_1^+\subseteq D$ and hence $0\in D$ which implies $D=M_n$, \\
if $a_i\in D$ and $\{a_1,\ldots,a_n\}\not\subseteq D$ then $a_i^+\not\subseteq D$. \\
The rest follows from the table above. Moreover, $\BDed\mathbf M_n$ is a $2^n$-element Boolean algebra since
\[
A\mapsto\left\{
\begin{array}{ll}
	M_n        & \text{if }A=\{a_1,\ldots,a_n\}, \\
	A\cup\{1\} & \text{otherwise}
\end{array}
\right.
\]
is an isomorphism from $(2^{\{a_1,\ldots,a_n\}},\subseteq)$ to $\BDed\mathbf M_n$.
\end{example}

The relationship between deductive systems and filters is described in the following results.

\begin{lemma}
	Let $\mathbf L=(L,\vee,\wedge,0,1)$ be a complemented lattice and $D$ a deductive system of $\mathbf L$. Then the following holds:
	\begin{enumerate}[{\rm(i)}]
		\item $D$ is an order filter of $\mathbf L$,
		\item if $x\to y\subseteq D$ for all $x,y\in D$ then $D$ is a filter of $\mathbf L$.
	\end{enumerate}
\end{lemma}

\begin{proof}
	Let $a,b\in L$.
	\begin{enumerate}[(i)]
		\item If $a\in D$ and $a\le b$ then $a\to b=1\in D$ and hence $b\in D$.
		\item According to (i), $D$ is an order filter of $\mathbf L$. If $a,b\in D$ then
\[
a\to(a\wedge b)=a^+\vee\big(a\wedge(a\wedge b)\big)=a^+\vee(a\wedge b)=a\to b\subseteq D
\]
showing $a\wedge b\in D$.
	\end{enumerate}
\end{proof}

If $\mathbf L$ is, moreover, modular then we can prove also the following.

\begin{proposition}
	Let $\mathbf L=(L,\vee,\wedge,0,1)$ be a complemented modular lattice and $F$ a filter of $\mathbf L$. Then $F$ is a deductive system of $\mathbf L$.
\end{proposition}

\begin{proof}
	If $a\in F$, $b\in L$ and $a\to b\subseteq F$ then according to Theorem~\ref{th2} (i) we have
	\[
		a\wedge b=a\wedge(a\to b)\subseteq F
	\]
	and due to $a\wedge b\le b$ we finally obtain $b\in F$.
\end{proof}

In the remaining part of this section we investigate when a given deductive system $D$ may induce an equivalence relation $\Phi$ such that $D=[1]\Phi$, i.e.\ $D$ being its kernel. We start with the following definition.

\begin{definition}\label{def1}
	For every complemented lattice $(L,\vee,\wedge,0,1)$ and every deductive system $D$ of $\mathbf L$ put
	\[
	\Theta(D):=\{(x,y)\in L^2\mid x\to y,y\to x\subseteq D\}.
	\]
\end{definition}

From Theorem~\ref{th1} (ii) we get that $\Theta(D)$ is reflexive and, by definition, it is symmetric.

It is easy to see that every congruence on a complemented modular lattice induces a deductive system.

\begin{proposition}\label{prop3}
Let $(L,\vee,\wedge,0,1)$ be a complemented modular lattice and $\Phi\in\Con(L,\wedge)$. Then the following holds:
\begin{enumerate}[{\rm(i)}]
	\item $[1]\Phi$ is a deductive system of $\mathbf L$,
	\item $\Theta([1]\Phi)\subseteq\Phi$.
\end{enumerate}
\end{proposition}

\begin{proof}
	Let $a,b\in L$.
	\begin{enumerate}[(i)]
		\item We have $1\in[1]\Phi$, and if $a\in[1]\Phi$ and $a\to b\subseteq[1]\Phi$ then according to Theorem~\ref{th2} {\rm(i)} we conclude
	\[
	b=1\wedge b\mathrel\Phi a\wedge b=a\wedge(a\to b)\subseteq[1\wedge1]\Phi=[1]\Phi.
	\]
	\item If $(a,b)\in\Theta([1]\Phi)$ then $a\to b,b\to a\subseteq[1]\Phi$ and hence again according to Theorem~\ref{th2} {\rm(i)} we obtain
	\[
	a=a\wedge1\mathrel\Phi a\wedge(a\to b)=a\wedge b=b\wedge a=b\wedge(b\to a)\mathrel\Phi b\wedge1=b
	\]
	showing $(a,b)\in\Phi$.
	\end{enumerate}
\end{proof}

That not all deductive systems arise in the way shown in Proposition~\ref{prop3} (i) can be seen as follows: According to Example~\ref{ex2}, $\{a_1,a_2,1\}$ is a deductive system of $\mathbf M_3$, but there does not exist some $\Phi\in\Con(M_3,\wedge)$ satisfying $[1]\Phi=\{a_1,a_2,1\}$ since this would imply $0=a_1\wedge a_2\in[a_1\wedge a_1]\Phi=[a_1]\Phi=\{a_1,a_2,1\}$, a contradiction.

The previous proposition shows that we need a certain compatibility of the induced relation $\Theta(D)$ with the lattice operations in order to show $D$ to be the kernel of $\Theta(D)$. For this sake, we define the following properties.

\begin{definition}
Let $(L,\vee,\wedge,0,1)$ be a complemented lattice and $\Phi$ an equivalence relation on $L$. We say that $\Phi$ has the {\em Substitution Property with respect to $^+$} if
\[
(a,b)\in\Phi\text{ implies }a^+\times b^+\subseteq\Phi,
\]
and the {\em Substitution Property with respect to $\to$} if
\[
(a,b)\in\Phi\text{ implies }(a\to c)\times(b\to c)\subseteq\Phi\text{ for all }c\in L.
\]
\end{definition}

Such an equivalence relation $\Phi$ can be related with the equivalence relation induced by its kernel $[1]\Phi$ and, moreover, this kernel is a deductive system.

\begin{theorem}\label{th3}
Let $\mathbf L=(L,\vee,\wedge,0,1)$ be a complemented lattice and $\Phi$ an equivalence relation on $L$ having the Substitution Property with respect to $\to$. Then the following holds:
\begin{enumerate}[{\rm(i)}]
	\item $\Phi$ has the Substitution Property with respect to $^+$,
	\item $[1]\Phi$ is a deductive system of $\mathbf L$,
	\item $\Phi\subseteq\Theta([1]\Phi)$.
\end{enumerate}	
\end{theorem}

\begin{proof}
	Let $a,b\in L$.
	\begin{enumerate}[(i)]
		\item According to Theorem~\ref{th1} (i), $(a,b)\in\Phi$ implies $a^+\times b^+=(a\to0)\times(b\to0)\subseteq\Phi$.
		\item If $a\in[1]\Phi$ and $a\to b\subseteq[1]\Phi$ then for every $x\in a\to b$ we have $(1,x)\in\Phi$ and according to Theorem~\ref{th1} (i) also $(x,b)\in(a\to b)\times(1\to b)\subseteq\Phi$ showing $b\in[1]\Phi$.
		\item According to Theorem~\ref{th1} (ii), $(a,b)\in\Phi$ implies
		\begin{align*}
			(a\to b)\times\{1\} & =(a\to b)\times(b\to b)\subseteq\Phi, \\
			(b\to a)\times\{1\} & =(b\to a)\times(a\to a)\subseteq\Phi.
		\end{align*}
	\end{enumerate}
\end{proof}

Now we are able to relate deductive systems with equivalence relations induced by them provided these deductive systems satisfy a certain compatibility condition defined as follows.

\begin{definition}
Let $\mathbf L=(L,\vee,\wedge,0,1)$ be a complemented lattice and $D$ a deductive system of $\mathbf L$. We call $D$ a {\em compatible deductive system} of $\mathbf L$ if it satisfies the following two additional conditions for all $a,b,c,d\in L$:
\begin{itemize}
	\item If $a\to b\subseteq D$ and $x\to(c\to d)\subseteq D$ for all $x\in a\to b$ then $c\to d\subseteq D$,
	\item if $a\to b,b\to a\subseteq D$ then $x\to(b\to c)\subseteq D$ for all $x\in a\to c$.
\end{itemize}
\end{definition}

Since the intersection of compatible deductive systems of $\mathbf L$ is again a compatible deductive system of $\mathbf L$, the set of all compatible deductive systems of $\mathbf L$ forms a complete lattice with respect to inclusion with top element $L$.

Now we show that also conversely as in Theorem~\ref{th3}, a compatible deductive system induces an equivalence relation having the Substitution Property with respect to $\to$.

\begin{theorem}
Let $\mathbf L=(L,\vee,\wedge,0,1)$ be a complemented lattice and $D$ a compatible deductive system of $\mathbf L$. Then the following holds:
\begin{enumerate}[{\rm(i)}]
	\item $\Theta(D)$ is an equivalence relation on $L$ having the Substitution Property with respect to $\to$,
	\item $[1]\big(\Theta(D)\big)=D$.
\end{enumerate}
\end{theorem}

\begin{proof}
	Let $a,b,c,d,e,f,g\in L$.
	\begin{enumerate}[(i)]
		\item As remarked after Definition~\ref{def1}, $\Theta(D)$ is reflexive and symmetric. Now assume $(a,b),(b,c)\in\Theta(D)$. Then $b\to a,a\to b\subseteq D$ and hence $x\to(a\to c)\subseteq D$ for all $x\in b\to c$. Because of $(b,c)\in\Theta(D)$ we have $b\to c\subseteq D$ and therefore $a\to c\subseteq D$. On the other hand $b\to c,c\to b\subseteq D$ which implies $x\to(c\to a)\subseteq D$ for all $x\in b\to a$ which together with $b\to a\subseteq D$ yields $c\to a\subseteq D$. This shows $(a,c)\in\Theta(D)$, i.e.\ $\Theta(D)$ is transitive. Now assume $(d,e)\in\Theta(D)$. Then $d\to e,e\to d\subseteq D$ and hence $x\to(e\to f)\subseteq D$ for all $x\in d\to f$. Because of $e\to d,d\to e\subseteq D$ we have $y\to(d\to f)\subseteq D$ for all $y\in e\to f$. Since $g\to A=\bigcup\limits_{x\in A}(g\to x)$ for all $A\subseteq L$, we have $x\to y,y\to x\subseteq D$ for all $(x,y)\in(d\to f)\times(e\to f)$ and hence $(x,y)\in\Theta(D)$ for all $(x,y)\in(d\to f)\times(e\to f)$ proving $(d\to f)\times(e\to f)\subseteq\Theta(D)$. Therefore $\Theta(D)$ has the Substitution Property with respect to $\to$.
		\item According to Theorem~\ref{th1} (i) and (ii) the following are equivalent: $a\in[1]\big(\Theta(D)\big)$; $a\to1,1\to a\subseteq D$; $1,a\in D$; $a\in D$.
	\end{enumerate}
\end{proof}

{\bf Author Contributions} Both authors contributed equally to this manuscript.

{\bf Funding} This study was funded by the Czech Science Foundation (GA\v CR), project 24-14386L, and IGA, project P\v rF~2024~011.

{\bf Data Availability} Not applicable.

{\bf Declarations}

{\bf Conflict of interest} The authors declare that they have no conflict of interest.



Authors' addresses:

Ivan Chajda \\
Palack\'y University Olomouc \\
Faculty of Science \\
Department of Algebra and Geometry \\
17.\ listopadu 12 \\
771 46 Olomouc \\
Czech Republic \\
ivan.chajda@upol.cz

Helmut L\"anger \\
TU Wien \\
Faculty of Mathematics and Geoinformation \\
Institute of Discrete Mathematics and Geometry \\
Wiedner Hauptstra\ss e 8-10 \\
1040 Vienna \\
Austria, and \\
Palack\'y University Olomouc \\
Faculty of Science \\
Department of Algebra and Geometry \\
17.\ listopadu 12 \\
771 46 Olomouc \\
Czech Republic \\
helmut.laenger@tuwien.ac.at

\end{document}